\documentclass[a4paper,reqno,11pt,oneside]{amsart}
\usepackage{graphicx} 
\usepackage{amsmath}
\usepackage{amsthm}
\usepackage{amssymb}
\usepackage{color}
\usepackage{xcolor}
\usepackage{bbm}
\theoremstyle{plain}
\newtheorem{thm}{Theorem}
  \theoremstyle{plain}
  \newtheorem{lem}[thm]{Lemma}
  \theoremstyle{plain}
  \newtheorem{cor}[thm]{Corollary}
   \newtheorem{prop}[thm]{Proposition}
  \theoremstyle{remark}
  \newtheorem{rem}[thm]{Remark}
  \newtheorem*{rem*}{Remark}
\def\div{\textup{div}}

\begin{document}
\title[Simplicity of eigenvalues]{Simplicity of eigenvalues for elliptic problems with mixed  Steklov-Robin boundary condition}

\author{Marco Ghimenti}
\address[Marco Ghimenti]{Dipartimento di Matematica
Universit\`a di Pisa
Largo Bruno Pontecorvo 5, I - 56127 Pisa, Italy}
\email{marco.ghimenti@unipi.it }

\author{Anna Maria Micheletti}
\address[Anna Maria Micheletti]{Dipartimento di Matematica
Universit\`a di Pisa
Largo Bruno Pontecorvo 5, I - 56127 Pisa, Italy}
\email{a.micheletti@dma.unipi.it }

\author{Angela Pistoia}
\address[Angela Pistoia] {Dipartimento SBAI, Universit\`{a} di Roma ``La Sapienza", via Antonio Scarpa 16, 00161 Roma, Italy}
\email{angela.pistoia@uniroma1.it}

\thanks{The first and the third authors are partiallty supported by INdAM - GNAMPA Project "Esistenza e propriet\`a qualitative per problemi di tipo Lane-Emden attraverso metodi asintotici" CUP E5324001950001}

\begin{abstract}  
    This paper investigates the spectral properties of two classes of elliptic problems characterized by mixed Steklov-Robin boundary conditions.  Our main objective is to prove that, for a generic domain, all the eigenvalues are simple. This result is established by employing domain perturbation techniques and analyzing the transversality of the associated operators.
\end{abstract}

\keywords{ Eigenvalues problem, Steklov-Robin-Neumann boundary conditions, generic properties, simplicity}

\subjclass{35J60, 35P05}
 \maketitle

\section{Introduction}

In the study of Partial Differential Equations (PDEs) —particularly the Helmholtz equation $\Delta u + \lambda u = 0$— the multiplicity of eigenvalues is deeply tied to the symmetry of the domain. An eigenvalue is said to be simple when, for a specific energy level or frequency, the system can take only one unique shape (eigenfunction), up to a scaling factor.The simplicity of eigenvalues is fundamental because it ensures that the qualitative, quantitative, and perturbative analysis of PDEs is stable, interpretable, and physically consistent. For this reason, proving the genericity of simple eigenvalues —i.e., that almost all domain shapes result in a simple spectrum— is a cornerstone of mathematical physics and engineering.\\

In this paper, we are interested in studying two specific eigenvalue problems:
\begin{equation}\label{Slosh-down}  
   \left\{
    \begin{array}{cc}
         -\Delta u=0& \text{ on } \Omega\\
        \partial_\nu u=\lambda u & \text{ in } S\\
        \partial_\nu u +u=0 & \text{ in }  W
    \end{array}
    \right.
\end{equation}
and 
\begin{equation}\label{Slosh-up}  
   \left\{
    \begin{array}{cc}
         -\Delta u+u=0& \text{ on } \Omega\\
        \partial_\nu u=\lambda u & \text{ in } S\\
        \partial_\nu u=0 & \text{ in }  W.
    \end{array}
    \right.
\end{equation}
Here $\Omega$ is an open bounded domain in $\mathbb{R}^n$, whose boundary $\partial \Omega $ is $C^{1,1}$ and it is composed of two connected parts $S$ and $W$, with non empty internal parts  $S^{\mathrm{o}}$ and $W^{\mathrm{o}}$ with  $S^{\mathrm{o}}\cap W^{\mathrm{o}}=\emptyset$. 
\\

Problem \eqref{Slosh-down} is a mixed Steklov-Robin eigenvalue problem for the Laplace operator, where harmonic functions are subject to Steklov conditions on $S$ and Robin conditions on $W$. This generalizes the classical Steklov problem, where the condition $\partial_\nu u = \lambda u$ is imposed on the entire boundary.\\
Problem \eqref{Slosh-up} is more established in the literature and fits naturally into the theory of Steklov-type problems for elliptic operators with mixed boundary conditions. Since the operator $-\Delta + \mathtt{I}$ with Neumann conditions on $W$ is coercive, this "sloshing-type" problem with a mass term represents a significant simplification compared to the classical sloshing problem
\begin{equation}\label{slosh}
   \left\{    \begin{array}{cc}
         -\Delta u =0& \text{ on } \Omega\\
        \partial_\nu u=\lambda u & \text{ in } S\\
        \partial_\nu u=0 & \text{ in }  W.
    \end{array}\right.
\end{equation}

There is an extensive literature on the spectral geometry of these problems (see, e.g., \cite{bkps,gp1,gp2,kkk,lpps}), providing variational characterizations, asymptotics, eigenvalue inequalities, and spectral monotonicity results.\\

In the present paper, we show that all eigenvalues of \eqref{Slosh-down} and \eqref{Slosh-up} are simple for generic domains $\Omega$. More precisely, we prove that for any such domain, there exists an arbitrarily small perturbation for which all eigenvalues $\lambda$ become simple. We also investigate whether such perturbations can leave one of the boundary components ($W$ or $S$) unchanged. Furthermore, we extend these results to versions of the above problems with non-constant coefficients, showing that simplicity is preserved under small perturbations of the potential. Finally, we consider an anisotropic elliptic equation $\partial_i(a_{ij}(x)\partial_j u)=0$ and prove that all eigenvalues are simple under generic perturbations of the coefficients $a_{ij}$.\\
To unify the notation, we define $d=\mathbbm{1}_S$, allowing the boundary conditions in \eqref{Slosh-down} and \eqref{Slosh-up} to be rewritten as $\partial_\nu u + u = \lambda d u$ (up to a shift $\lambda \to \lambda+1$) and $\partial_\nu u = \lambda d u$, respectively. The inclusion of the linear term is motivated by our proof technique and is discussed further in the context of the variational structure. This approach significantly simplifies the analysis compared to the zero-mass case with pure Neumann conditions on $W$; a dedicated study of the sloshing problem \eqref{slosh} is currently in preparation.

Regarding domain perturbations, we consider maps $\psi \in \mathcal{D} := \{\psi \in C^2(\mathbb{R}^n, \mathbb{R}^n) : \|\psi\|_{C^2} < 1/2\}$, defining the perturbed domain as $\Omega_\psi := (I+\psi)\Omega$, where $I$ denotes the identity map on $\mathbb{R}^n$. Furthermore, for $\xi \in \partial \Omega_\psi$, we define $d_\psi(\xi) := d(x)$, where $\xi = (I+\psi)(x)$."

Our main results concerning domain perturbations are as follows.

\begin{thm}\label{thm:main-simp}
For any $\varepsilon >0$, there exists $\psi \in \mathcal{D}$, with 
$\|\psi\|_{C^2}<\varepsilon$, for which  all the eigenvalues of  the problem
\begin{equation}\label{mainpsisimp}
   \left\{
    \begin{array}{cc}
         -\Delta u=0& \text{ on } \Omega_\psi\\
        \partial_\nu u +u=\lambda d_\psi(\xi) u & \text{ in } \partial\Omega_\psi\\
    \end{array}
    \right.
\end{equation} 
are simple. Also, $\psi$ can be chosen to leave one of the two components of the boundary, $S$ or $W$, unchanged.

\end{thm}

\begin{thm}
For any $\varepsilon >0$, there exists $\psi \in \mathcal{D}$, with $\|\psi\|<\varepsilon$,  for which all the eigenvalues of  the problem
\begin{equation}\label{2psisimp}
   \left\{
    \begin{array}{cc}
         -\Delta u +u=0& \text{ on } \Omega_\psi\\
        \partial_\nu u=\lambda d_\psi(\xi) u & \text{ in } \partial\Omega_\psi\\
    \end{array}
    \right.
\end{equation} 
are simple. Also, $\psi$ can be chosen to leave $W$ unchanged.
\label{thm:2-simp}
\end{thm}

Regarding perturbations of the coefficients, we first consider a positive potential $a(x)$ and study the following problems:     
\begin{equation}\label{pot-down}  
   \left\{
    \begin{array}{cc}
         -\Delta u=0& \text{ on } \Omega\\
         \partial_\nu u+ au=d\lambda u & \text{ in } \partial \Omega
    \end{array}.
    \right.
\end{equation}
where $a\in L^{n-1}(\partial\Omega)$,
  and 
  \begin{equation}\label{pot-up}  
   \left\{
    \begin{array}{cc}
         -\Delta u+a(x)u=0& \text{ on } \Omega\\
        \partial_\nu u=d\lambda u & \text{ in } \partial \Omega
    \end{array},
    \right.
\end{equation}
for $a\in L^{n/2}(\Omega)$,
obtaining the following result

\begin{thm}
\label{thm:pot}
For any $\varepsilon>0$ there exists $b\in C^{0}(\Omega)$, with $\|b\|_{C^0}<\varepsilon$, 
  such that all the eigenvalues of  the problem  (\ref{pot-down}) or (\ref{pot-up}), with coefficient $\tilde a=a+b$ are simple.
\end{thm}

Finally we address   an anisotropic elliptic operator in divergence form $Lu:=\partial_i (a_{ij}(x)\partial_j u)$ such that $a_{ij} \in C^1(\Omega)$ and $A(x)=(a_{ij})_{ij}$ is uniformly positive definite. We consider

\begin{equation}\label{div-down}  
   \left\{
    \begin{array}{cc}
         L u=0& \text{ on } \Omega\\
         \partial_\nu u+ u=d\lambda u & \text{ in } \partial \Omega
    \end{array}.
    \right.
\end{equation}
and we obtain

\begin{thm}
\label{thm:div}
For any $\varepsilon>0$ there exists $B=(b_{ij})_{ij}\in C^{1}(\Omega)$, with $\|B\|_{C^1}<\varepsilon$, 
  such that all the eigenvalues of  the problem (\ref{div-down}), with coefficient $\tilde A=A+B$ are simple.
\end{thm}
By applying the same methodology, we obtain a similar result for the following case (proof omitted)
\begin{cor}\label{thm:div-up}
Given $A=(a_{ij})_{ij}$ a $C^1(\Omega)$ positive definite matrix, for any $\varepsilon>0$ there exists $B=(b_{ij})_{ij}\in C^{1}(\Omega)$, with $\|B\|_{C^1}<\varepsilon$, 
  such that all the eigenvalues of  the problem 
  \begin{equation} 
   \left\{
    \begin{array}{cc}
        \partial_i \Big((a_{ij}+b_{ij})\partial_j u\Big) +u=0& \text{ on } \Omega\\
         \partial_\nu u=d\lambda u & \text{ in } \partial \Omega
    \end{array}.
    \right.
\end{equation}
\end{cor}

All these results are obtained through a suitable application of Micheletti’s approach to the simplicity of eigenvalues (see, e.g., \cite{mi1,mi2,Lupo}). The fundamental tool for this analysis is the following abstract result (see \cite{mi2, fgmp}), which provides a "non-splitting" condition for eigenvalues under operator perturbation.

\begin{thm}\label{thm:astratto} 
Let $T_{b}:X\rightarrow X$ be a compact, self-adjoint operator on a Hilbert space $X$, which depending smoothly on a parameter $b$ in a real Banach space $B$. 
Assume that $T_{b}$ is Frechét differentiable in $b=0$. Let $\bar{\lambda}$ an eigenvalue for $T_0$ with multiplicity $m>1$  and let $x_{1}^{0},\dots,x_{m}^{0}$ form an orthonormal basisof the corresponding eigenspace. If  $\lambda(b)$ is an eigenvalue for $T_{b}$ such that $\lambda(0)=\bar\lambda$ and  $\lambda(b)$ maintains multiplicity $m$ for all $b$ with $\|b\|_{C^{0}}$ small, then for all such $b$ there exist a $\rho=\rho(b)\in\mathbb{R}$
for which
\begin{equation}
\left\langle T'(0)[b]x_{j}^{0},x_{i}^{0}\right\rangle _{X}=\rho\delta_{ij}\text{ for }i,j=1,\dots,m.
\label{eq:spezzamentoastratto}
\end{equation}
\end{thm}
\subsection{Notations}\label{sec:not}
We set here some short notations for integral quantities which will be fundamental in the proof of Theorem \ref{thm:main-simp} 
\begin{eqnarray*}
    L_1(u,\varphi)&:=&\int_\Omega-\psi_t \frac{\partial}{\partial x_t} \left(\nabla u\nabla \varphi\right) + \psi_i\frac{\partial }{ \partial x_j}\left(\frac{\partial u}{\partial x_i}\frac{\partial \varphi}{\partial x_j}+
\frac{\partial u}{\partial x_j}\frac{\partial \varphi}{\partial x_i}
\right) dx\\
  I_1(u,\varphi)&:=&\int_{\partial\Omega}\psi_t \nu_t \nabla u\nabla \varphi -
\psi_i \nu_j\left(\frac{\partial u}{\partial x_i}\frac{\partial \varphi}{\partial x_j}+\frac{\partial u}{\partial x_j}\frac{\partial \varphi}{\partial x_i}
\right) d\sigma\\
  I_2( u,\varphi)&:=&\int_{\partial \Omega} u\varphi \left(\div \psi -\sum_{r=1}^n \partial_\nu \psi_r \nu_r\right) d\sigma\\
   I_3( u,\varphi)&:=&  
\int_\Omega-\psi_t \frac{\partial}{\partial x_t}  u \varphi  dx
+\int_{\partial\Omega}\psi_t \nu_t  u \varphi d\sigma
   \\
      G(u,\varphi)&:=&\int_{ S} u\varphi \left(\div \psi -\sum_{r=1}^n \partial_\nu \psi_r \nu_r\right) d\sigma
 \end{eqnarray*}

\section{Variational setting and proof of Theorem \ref{thm:main-simp}}\label{sec:sloshsimp}

In this section we address equation (\ref{Slosh-down}).
As stated in the introduction, this equation is a modification of the original sloshing problem (\ref{slosh}), and the inclusion of a linear term $u$ in the Neumann condition on the boundary overcomes a technical obstacle.  Specifically, this boundary formulation allows us to define on $H^1(\Omega)$ the scalar product
\[
a_{\Omega}(u,v):=\int_\Omega \nabla u \nabla v \ dx +\int_{\partial \Omega} uv\ d\sigma .
\]
which equips $H^1(\Omega)$ with the equivalent norm $\| u\|_{H^1}=\|u\|:=\big(a_{\Omega}(u,v)\big)^{1/2}$. 

Problem  (\ref{Slosh-down}) admits a weak formulation: a function $e$ is an eigenfunction with corresponding eigenvalue $\lambda$ if and only if 
\begin{equation}\label{eq:weaksimp}
    a_{\Omega}(e,\varphi)=\lambda\int_{\partial \Omega} d(x)e\varphi\ d\sigma .
\end{equation}
for all $\varphi\in H^1(\Omega)$. We recall here that $d(x)=1$ if $x\in S$ and $d(x)=0$ if $x\in W$.

Given $f\in H^1(\Omega)$, we consider the continuous linear functional on $H^1(\Omega)$ given by
$$
L_f(\varphi):=\int_{\partial\Omega} d(x)f\varphi\  d\sigma,
$$
and, by Riesz theorem, for any $f\in H^1(\Omega)$, there exists a unique $w\in H^1(\Omega)$, such that 
\[
a_{\Omega}(w,\varphi)=L_f (\varphi)=\int_{\partial \Omega} d(x)f\varphi\  d\sigma .
\]
Now, let us define $E_\Omega$ as the operator, from $ H^1(\Omega)$ in itself, which associates $f$ to ,$w$, i.e. 
\begin{equation}\label{eq:Eusimp}
 a_{\Omega}(w,\varphi)=a_{\Omega}(E_\Omega(f),\varphi)=\int_{\partial \Omega} d(x)f\varphi\  d\sigma 
\end{equation}
for all $\varphi \in H^1 (\Omega)$. So,  $(\lambda, e)$ are an eigenpair of problem (\ref{Slosh-down}), in a weak sense, if and only if $e=E_\Omega(\lambda e)$, or, equivalently, $e$ is an eigenfunction of $E_\Omega$ with eigenvalue $\mu=1/\lambda$.

As $\{\lambda_k\}_k$ is a sequence of positive increasing eigenvalues for problem 
(\ref{Slosh-down}), the reciprocals  $\{\mu_k:=1/\lambda_k\}_k$  form a decreasing sequence of positive eigenvalues of $E_\Omega $ decreasing to zero while $k$ increases. Also, any $\mu_k$ admits the following Min-Max characterization:
\begin{eqnarray*}  
\mu_1= \max_{a_\Omega(\varphi,\varphi)=1}\int_{\partial\Omega} d(x)\varphi^2
&\quad&
\mu_t= 
\inf_{[v_1,\dots,v_{t-i}]}
\sup_{
\begin{array}{c}
a_\Omega(\varphi,v_i)=0\\
a_\Omega(\varphi,\varphi)=1
\end{array}}\int_{\partial\Omega} d(x)\varphi^2.
\end{eqnarray*}
In this way we transformed the original problem in an eigenvalue problem for the operator $E_\Omega$. 

This entire construction carries over to any perturbed domain $\Omega_\psi$. However, in order to use Theorem \ref{thm:astratto} in the proof,  it is necessary to work within the fixed space ${H^1}(\Omega)$. To that end we pull back any function defined on $\Omega_\psi$ on the unperturbed domain $\Omega$ via the diffeomorphism $(I+\psi)$.  

Specifically, if $x\in \Omega $ and $\xi=(I+\psi)(x)=x+\psi(x)$, given a function $\tilde u\in  H^1(\Omega_\psi)$, we define the corresponding function $u\in {H^1}(\Omega)$ by
 $$
\tilde u(\xi)=\tilde u(x+\psi (x)):=u(x).
$$
This defines a map 
\begin{align*}
\gamma_\psi&:{H^1}(\Omega_{\psi})\rightarrow {H^1}(\Omega);\\
\gamma_\psi(\tilde{u})&:=u(x)=\tilde{u}(x+\psi(x))=\tilde u\circ (I+\psi)(x).
\end{align*}
If $\|\psi\|_{C^2}<1/2$, the map $I+\psi$ is invertible with inverse $(I+\psi)^{-1}=I+\chi $ and the following maps are continuous isomorphisms:
\begin{eqnarray}
 \label{gamma-psi}
\gamma_{\psi} :H^1(\Omega_{\psi})\rightarrow H^1(\Omega)
& &
\gamma_{\psi}^{-1}=\gamma_{\chi}:H^1(\Omega)\rightarrow H^1(\Omega_{\psi}).
\end{eqnarray}
Now, given $\tilde u,\tilde \varphi \in H^1(\Omega_\psi)$ we define a bilinear form $\mathcal{A}$  on $H^1(\Omega)$ by pulling back the bilinear form  $a_{\Omega_\psi} $ on $H^1(\Omega_\psi)$ as follows.
For any $\tilde u,\tilde \varphi\in H^1 (\Omega_\psi)$ set
\begin{multline}\label{def-Epsi}
    a_{\Omega_\psi}(\tilde u,\tilde \varphi)=
   \int_{ \Omega_\psi}\nabla \tilde u\nabla\tilde \varphi\  dx
    +\int_{\partial \Omega_\psi}\tilde u\tilde \varphi\ d\sigma 
    \\
    =\mathcal{A}_\psi (u,\varphi)
    +\int_{\partial \Omega} u \varphi B_\psi d\sigma =:\mathcal{E}_\psi (u,\varphi) 
\end{multline}
where $u=\gamma_{\psi}\tilde u$, $\varphi=\gamma_{\psi} \tilde \varphi$, $B_\psi$  represents the Jacobian  of the boundary change of variables, and
\begin{equation}
    \label{defA}
\mathcal{A}_\psi (u,\varphi) :=\int_{\Omega_\psi} \nabla \tilde u \nabla \tilde \varphi \ dx.  
\end{equation}
  At this point, we may express the counterpart of (\ref{eq:Eusimp}): if $\tilde u=E_{\Omega_\psi}(\tilde u)$, then
 \begin{equation}\label{eq:fond1simp}
     a_{\Omega_\psi}(E_{\Omega_\psi}(\tilde u),\tilde \varphi)=\int_{\partial {\Omega_\psi}} d_\psi(\xi) \tilde u \tilde \varphi\  d\sigma =\int_{\partial \Omega} d(x)u\varphi\ B_\psi d\sigma 
 \end{equation}
where $d_\psi(\xi)=d_\psi(x+\psi(x))=d(x)$ for $\xi:=(x+\psi(x))\in\partial\Omega_\psi$ and for $x\in \partial \Omega$. Also, 
\begin{equation}\label{eq:fond2simp}
       a_{\Omega_\psi}(E_{\Omega_\psi}(\tilde u),\tilde \varphi)= 
       a_{\Omega_\psi}( E_{\Omega_\psi}(\gamma_\psi^{-1} u),\gamma_\psi^{-1}  \varphi)
=\mathcal{E}_\psi (\gamma_\psi E_{\Omega_\psi}(\gamma_\psi^{-1} u),\varphi).
\end{equation}
Let us define $T_\psi:= \gamma_\psi E_{\Omega_\psi}\gamma_\psi^{-1} $ , $T_\psi :H^1(\Omega)\rightarrow H^1(\Omega)$. We have that $T_0=E_\Omega $,  and, combining (\ref{eq:fond1simp}) and (\ref{eq:fond2simp}), it holds
\begin{equation}\label{fondsimp}
    \mathcal{E}_\psi (T_\psi u,\varphi)=\int_{\partial \Omega} d(x)u\varphi\ B_\psi d\sigma .
\end{equation} 
By the computational lemmas in Appendix \ref{Sec:der} we now possess all necessary components 
to differentiate equation (\ref{fondsimp}) with respect to the $\psi$ variable, at $\psi=0$, obtaining the principal result of this section
 \begin{cor}
 It holds
 \begin{equation*}
    \mathcal{E}_\psi'(0)[\psi] (T_0 u,\varphi)+
      \mathcal{E}_0 (T_\psi '(0)[\psi] u,\varphi)=\left. \frac{\partial}{\partial \psi}\int_{\partial \Omega}d(x)  u\varphi\ B_\psi d\sigma\right|_{\psi=0}.
\end{equation*} 
In addition, in light of computation of Appendix \ref{Sec:der}, we have
\begin{multline}\label{cond1}
   \mathcal{E}_0 (T_\psi '(0)[\psi] u,\varphi)= a_\Omega  (T_\psi '(0)[\psi] u,\varphi)=\\
   =\int_\Omega\psi_t \frac{\partial}{\partial x_t} \left(\nabla T_0u\nabla \varphi\right) - \psi_i\frac{\partial }{ \partial x_j}\left(\frac{\partial T_0u}{\partial x_i}\frac{\partial \varphi}{\partial x_j}+
\frac{\partial T_0u}{\partial x_j}\frac{\partial \varphi}{\partial x_i}
\right) dx\\ 
-    \int_{\partial\Omega}\psi_t \nu_t \nabla T_0u\nabla \varphi -
\psi_i \nu_j\left(\frac{\partial T_0u}{\partial x_i}\frac{\partial \varphi}{\partial x_j}+\frac{\partial T_0u}{\partial x_j}\frac{\partial \varphi}{\partial x_i}
\right) d\sigma\\
 -   \int_{\partial \Omega} T_0u\varphi \left(\div \psi -\sum_{r=1}^n \partial_\nu \psi_r \nu_r\right) d\sigma+\int_{S}  u\varphi \left(\div \psi -\sum_{r=1}^n \partial_\nu \psi_r \nu_r\right)d\sigma  \\
    =:-L_1( T_0u, \varphi)-
I_1( T_0u, \varphi)
-I_2(T_0 u, \varphi)+G(u,\varphi),
\end{multline}
where $L_1$, $I_1$, $I_2$ and $G$ are defined in the notation section.
\end{cor}
\subsection{The no splitting condition}
By the previous corollary we are in position to apply abstract Theorem \ref{thm:astratto} to Problem (\ref{mainpsisimp}). In fact, suppose that $\mu $ is an eigenvalue of $T_0$ with multiplicity $m>1$, and let $\{e_1, \dots,e_m \}$ be an orthonormal base for its eigenspace. If any perturbation $\psi$ leaves the multiplicity of $\mu$ unchanged, then, as in equation (\ref{eq:spezzamentoastratto}) we have to compute $   a_\Omega  (T_\psi '(0)[\psi] u,\varphi)$ and check that, for any $\psi$, there exists $\rho=\rho(\psi)\in \mathbb{R}$ such that
\begin{equation}
       a_\Omega  (T_\psi '(0)[\psi] e_r,e_s)=  \rho(\psi)\delta_{rs}.
\end{equation}
Having in mind (\ref{cond1}), and using that $T_0e_r=\mu e_r$ for all $r=1,\dots,m$, 
this can be formulated as
\begin{equation}
    \label{cond2}
    -L_1(\mu e_r,e_s)-
I_1(\mu  e_r,e_s)
-I_2(\mu  e_r, e_s)+G( e_r,e_s)=\rho\delta_{rs}.
\end{equation}
Equation (\ref{cond2}) represents a {\em nosplitting condition} for Problem  (\ref{mainpsisimp}). We will see that is possible to find a perturbation which violates condition  (\ref{cond2}).

To this aim, consider firstly a perturbation $\psi$ compactly supported in $\Omega$, so that $\psi\equiv0$ on the boundary.  In this case $I_1$, $I_2$, and $G$ vanish, and condition (\ref{cond2}) reads as

\begin{multline*}
    L_1( \mu e_r,e_s)=
  \mu \int_\Omega-\psi_i \frac{\partial}{\partial x_i} \nabla e_r\nabla e_s + \psi_i\frac{\partial }{ \partial x_j}\left(\frac{\partial e_r}{\partial x_i}\frac{\partial e_s}{\partial x_j}+
\frac{\partial e_r}{\partial x_j}\frac{\partial e_s}{\partial x_i}
\right) dx
\\=-\mu \delta_{rs}
\end{multline*}
Since this holds for arbitrary $\psi$, and  we may choose $\psi$ such that only one component $\psi_i$  is nonzero  at a time, then, for any $i$,  if $r\neq s$
\begin{equation*}
    \frac{\partial}{\partial x_i} \nabla e_r\nabla e_s +\sum_j\frac{\partial }{ \partial x_j}\left(\frac{\partial e_r}{\partial x_i}\frac{\partial e_s}{\partial x_j}+
\frac{\partial e_r}{\partial x_j}\frac{\partial e_s}{\partial x_i}
\right)=0
\end{equation*}
 almost everywhere on $\Omega$ and, when $r=s$, 
 \begin{equation*}
     \frac{\partial}{\partial x_i} |\nabla e_1|^2 +2\sum_j\frac{\partial }{ \partial x_j}\left(\frac{\partial e_1}{\partial x_i}\frac{\partial e_1}{\partial x_j}
\right)=\cdots=
\frac{\partial}{\partial x_i} |\nabla e_\nu|^2 +2\sum_j\frac{\partial }{ \partial x_j}\left(\frac{\partial e_\nu}{\partial x_i}\frac{\partial e_\nu}{\partial x_j}
\right)
 \end{equation*}
 That is means that, for any $\psi$ (not necessarily compactly supported in $\Omega$),  then $L_1(\mu e_r,e_s)=C(\psi)\delta_{rs}$, and we can absorb $L_1$  in the right hand side of (\ref{cond2}). Hence, the no splitting condition becomes
 \begin{equation}{\label{nosplitsimp}}
    -I_1( \mu e_r,e_s)
-I_2( \mu e_r, e_s)+G(e_r,e_s)=C(\psi)\delta_{rs}.
\end{equation}

\begin{rem}\label{uniquecont} To  study  the {\em no-splitting} condition (\ref{nosplitsimp}) it is useful to remind that
 if an eigenfunction $e$ and its normal derivative $\partial_\nu e$ vanish in a neighbourhood of a boundary point $\xi$, which does not belong to the interface $\Gamma:=S\cap W$, then  $e$ is identically zero.
 Indeed,
if   $\Omega$ is a  $C^{1,1}-$domain the eigenfunction $e$ belongs to  $H^2(S)\cap H^2(W)$ (see for example \cite[Theorem 2.4.2.7]{grisvard}). 
Moreover, by the   uniqueness in the Cauchy Problem as stated in \cite[p. 83]{henry}, we deduce that $e$ identically vanishes in a neighborhood
 $\xi$ and  by the Unique Continuation Principle we deduce that it vanishes anywhere in $\Omega$.

\end{rem}

In the next subsections we show that there exists perturbation which leaves one of the two components of the boundary fixed, for which (\ref{nosplitsimp}) does not hold.

\subsection{Perturbations supported in  $S$}\label{sec:Sneumann}

Now, we take a perturbation which leaves $W$ fixed. We immediately have that $I_2( \mu e_r, e_s)=\mu G( e_r, e_s)$, so the no splitting condition simplifies as
\begin{equation}
    -I_1(  \mu e_r,e_s)+(1-\mu) G( e_r, e_s)
=C(\psi)\delta_{rs}, 
\end{equation}
that is
  \begin{multline}\label{eq:splitS0}
    - \mu\int_{S}\psi_t\nu_t\nabla e_r\nabla e_s d\sigma
    + \mu \int_{S}
\psi_i\nu_j \left(\frac{\partial e_r}{\partial x_i}\frac{\partial e_s}{\partial x_j}+\frac{\partial e_r}{\partial x_j}\frac{\partial e_s}{\partial x_i}
\right) d\sigma\\
+(1-\mu)  \int_{S}e_r e_s \left(\div \psi -\sum_{r=1}^n \partial_\nu \psi_r \nu_r\right)d\sigma \ =
C(\psi)\delta_{rs}.
 \end{multline}
Notice that  $\nu_j\frac{\partial e_r}{\partial x_j}=\partial_\nu e_r=(\lambda-1) e_r$, so integrating by parts we get, for the second term of (\ref{eq:splitS0}), 
\begin{multline}
    \mu \int_{S}
\psi_i\nu_j \left(\frac{\partial e_r}{\partial x_i}\frac{\partial e_s}{\partial x_j}+\frac{\partial e_r}{\partial x_j}\frac{\partial e_s}{\partial x_i}
\right) d\sigma=\\
\mu(\lambda-1) \int_{S}
\psi_i\frac{\partial }{\partial x_i}\left(e_r e_s
\right) d\sigma=-\frac{\lambda-1}\lambda 
\int_{S}
e_r e_s\div\psi
 d\sigma,
\end{multline}
and, since in the third integral of (\ref{eq:splitS0}), $(1-\mu)=(\lambda-1)/\lambda$, the no splitting condition becomes

  \begin{multline}\label{eq:splitS1}
    - \frac 1\lambda\int_{S}\psi_t\nu_t\nabla e_r\nabla e_s d\sigma
+\frac{\lambda-1}\lambda   \int_{S}e_r e_s \left(\sum_{r=1}^n \partial_\nu \psi_r \nu_r\right)d\sigma \ =
C(\psi)\delta_{rs}.
 \end{multline}
There exists at least one component $\nu_t \neq 0$. For the sake of simplicity, let us take $\nu_1\neq 0$.
Now choose $\psi=(\psi_1,0,\dots,0)$ with $\psi_1$ arbitrarily defined on $S$ and such that $\frac{\partial \psi_1} {\partial \nu}=0$. Thus  $\nabla e_r \nabla e_s=0$ if $r\neq s$ and $|\nabla e_1|^2=\dots=|\nabla e_m|^2$ almost everywhere on $S$ and the term $\int_{S}\psi_t\nu_t \nabla e_r\nabla e_s d\sigma$ can be absorbed in the left hand side of the condition and we remain with
\begin{equation}
     C(\psi)\delta_{rs}=
     +\frac{\lambda-1}\lambda \int_{ S}  e_r  e_s\sum_{r=1}^n \partial_\nu \psi_r \nu_r d\sigma
\end{equation}
Since we can choose $\sum_{r=1}^n \partial_\nu \psi_r \nu_r $ as and arbitrary function on $S$, we get that $ e_r  e_s=0$ if $r\neq s$ and $| e_1|^2=\dots=| e_m|^2$ almost everywhere on $S$. This implies that on $S$ we get, for any $r=1,\dots, m$
\begin{eqnarray}
    e_r=0;&& \frac{\partial e_r}{\partial \nu}=(\lambda-1) e_r=0
\end{eqnarray}
which leads to a contradiction by the unique continuation principle (see Remark \ref{uniquecont}).

\subsection{Perturbations supported in $W$}

If $\psi$ leaves $S$ fixed, namely,  the support of $\psi$ intersects $\partial \Omega$ only within the set $W$, we have that $G( e_r,e_s)\equiv 0$. Proceeding as before, we get that the no splitting condition (\ref{nosplitsimp}) becomes
  \begin{multline}\label{eq:nospiltDir}
 C(\psi)\delta_{rs}=-\mu\int_{W}\psi_t \nu_t \nabla e_r\nabla e_s
 d\sigma
+\mu\int_{W}
\psi_t\left(\frac{\partial e_r}{\partial \nu}\frac{\partial e_s}{\partial x_t}+\frac{\partial e_r}{\partial x_t}\frac{\partial e_s}{\partial \nu}
\right) d\sigma \\
-\mu \int_{W} e_r e_s \left(\div \psi -\sum_{r=1}^n \partial_\nu \psi_r \nu_r\right) d\sigma
\end{multline}
Now, since $\partial_\nu e_r+e_r=0$, on $W$ (and the same holds for $e_s$)  we have $\left(\frac{\partial e_r}{\partial \nu}\frac{\partial e_s}{\partial x_t}+\frac{\partial e_r}{\partial x_t}\frac{\partial e_s}{\partial \nu}
\right)=\left(- e_r\frac{\partial e_s}{\partial x_t}-\frac{\partial e_r}{\partial x_t} e_s
\right)=-\frac{\partial }{\partial x_t}(e_r e_s)$, and, integrating by parts, equation (\ref{eq:nospiltDir}) reads as
\begin{multline*}
 C(\psi)\delta_{rs}=-\mu\int_{W}\psi_t \nu_t \nabla e_r\nabla e_sd\sigma
 +\mu\int_{W}
\frac{\partial \psi_t}{\partial x_t}e_r e_s d\sigma 
\\-\mu \int_{W} e_r e_s \left(\div \psi -\sum_{r=1}^n \partial_\nu \psi_r \nu_r\right) d\sigma\\
=-\mu\int_{W}\psi_t \nu_t \nabla e_r\nabla e_sd\sigma
+\mu \int_{W} e_r e_s \left( \sum_{r=1}^n \partial_\nu \psi_r \nu_r\right) d\sigma
\end{multline*}
Without loss of generality we may assume that $\nu_1 \neq 0$ in the support of $\psi$.  We then  choose $\psi=(\psi_1,0,\dots,0)$, with $\psi_1$ arbitrarily chosen in $W$. Hence, the nosplitting condition simplifies to 
\begin{equation*}
    -\mu\int_{W}\psi_1 \nu_1 \nabla e_r\nabla e_sd\sigma
+\mu \int_{W} e_r e_s  \partial_\nu \psi_1 \nu_1 d\sigma=C(\psi)\delta_{rs}
\end{equation*}
We now extend  $\psi_1$ in a neighborhood of $W$
in such a way that it satisfies  $\partial_\nu \psi_1=0$. Since $\psi_1$ remains arbitrary within  $W$, we deduce that $\nabla e_r\nabla e_s=0$ for $r\neq s$ and $|\nabla e_1|^2=\dots=|\nabla e_m|^2$ almost everywhere on $W$. Therefore, the no splitting condition reduces to:
\begin{equation}
     C(\psi)\delta_{rs}=\mu \int_{W} e_r e_s \left( \sum_{r=1}^n \partial_\nu \psi_r \nu_r\right) d\sigma
\end{equation}
and since $ \sum_{r=1}^n \partial_\nu \psi_r \nu_r$ is arbitrary,  it follows once again that $e_j=\partial_\nu e_j=0$ on $W$, for any $j=1,\dots,m$, leading to a contradiction.

\subsection{Proof of Theorem \ref{thm:main-simp}}
In this subsection we finally give the proof of the first of the main results of this paper. As claimed in the introduction, we use the same strategy of \cite[Proof of Thm 1]{fgmp},\cite[Sec 5]{MiSNS}, and \cite[Lemma 7]{mi2}, which we refers to for all the technical details which we will omit hereafter for the sake of readability
 
We start proving that we can split a single eigenvalue of multiplicity $m$ in several distinct  eigenvalues each with multiplicity strictly less then $m$. Then, by iterating the procedure, we can find a perturbation for which this eigenvalue splits in $m$ simple eigenvalue, and finally prove  Theorem \ref{thm:main-simp}.

\begin{prop}
\label{thm:main-tool} Let $\bar \mu$ an eigenvalue for the operator $T_0$ with
multiplicity $m>1$. Let $U$ and open bounded interval such
that 
\[
\bar{U}\cap\sigma\left(T_0\right)=\left\{ \bar{\mu}\right\} ,
\]
where $\sigma(T_0)$  represents  the spectrum of $T_0$. 

Then, there exists $\psi\in \mathcal{D}$, and $\bar\varepsilon>0$  such that with $\|\psi\|_{\mathcal{D}}\le\bar\varepsilon$ and
\[
\bar{U}\cap\sigma(T_\psi))=\left\{ \mu_{1}^{{\Omega_\psi}},\dots,\mu_{k}^{{\Omega_\psi}}\right\} ,
\]
where $1<k\le m$ and, for all $t=1,\dots, k$, any eigenvalue $\mu_{t}^{{\Omega_\psi}}$ has multiplicity $m_t<m$ , with $\sum_{t=1}^{k}m_{t}=m$.
Also, we can choose $\psi$ such which leaves $S$ unchanged,  or, reversely, which  leaves $W$ unchanged.
\end{prop}
\begin{proof}
We recall that if $\|\psi\|_{\mathcal{D}}$ is small, the multiplicity of an eigenvalue $\mu^{\Omega_\psi}$ near $\bar{\mu}$ is lesser or equal than the multiplicity of $\bar{\mu}$.
In the previous subsection we have proved that there exists at least a small perturbation, which leaves one of the two components of the boundary unchanged and for which the multiplicity in not maintained. So, the multiplicity decreases under this perturbation.
\end{proof}

The next corollary follows from Proposition \ref{thm:main-tool}, composing
a finite number of perturbations.
\begin{cor}\label{cor:auto-semplice-0}
Let $\bar\lambda$ an eigenvalue for Problem (\ref{Slosh-down}) with multiplicity $m$.  For any $\varepsilon>0$ sufficiently small there exist $\psi \in \mathcal{D}$  with $\|\psi\|_\mathcal{D}<\varepsilon$, compactly supported in $S$, or in $W$,  for which  Problem (\ref{mainpsisimp})  has exactly $m$ simple eigenvalues in a neighborhood of $\bar\lambda$ .

\end{cor}
At this point we are in position to prove the main result of this paper.
\begin{proof}[Proof of Theorem  \ref{thm:main-simp}]
To prove the result, we iterate the procedure of Proposition \ref{thm:main-tool} and Corollary \ref{cor:auto-semplice-0} countably many times.  
Specifically, suppose that  $ \lambda_{q_1}$, for some $q_1\in \mathbb{N}$ is the first multiple eigenvalue, with multiplicity $m_1>1$, for problem (\ref{Slosh-down}) on $\Omega$. Then, by Proposition  \ref{thm:main-tool} and Corollary  \ref{cor:auto-semplice-0}, there exists $\varepsilon_1$ and a perturbation $\psi_1\in\mathcal{D}$ with  $\|\psi_1\|_\mathcal{D}<\varepsilon_1$, such that for Problem (\ref{mainpsisimp}) in $\Omega_{1}:=(I+\psi_1)\Omega$ all the eigenvalues $\lambda^{\Omega_1}_1, \lambda^{\Omega_1}_1\dots,\lambda^{\Omega_1}_{q_1},\dots \lambda^{\Omega_1}_{{q_1}+m_1-1}$ are simple. Let us set $F_1:=(I+\psi_1)$. 

If all the eigenvalues of Problem (\ref{mainpsisimp}) on $\Omega_1$ are simple, the proof is complete. Otherwise, let $ \lambda^{\Omega_1}_{q_2}$, with $q_2\ge q_1+m_1$ be the first multiple eigenvalue with multiplicity $m_2>1$.  In this case, we can find another perturbation $\psi_2$, with $\|\psi_2\|_\mathcal{D}<\varepsilon_2< \varepsilon_1/2$ which splits the eigenvalue  $ \lambda^{\psi_1}_{q_2}$ while leaving the multiplicity of the previous eigenvalues unchanged. More precisely, setting $F_2:=(I+\psi_2)$, we define the new perturbed domain $\Omega_2=F_2(\Omega_1)=F_2\circ F_1 (\Omega)$ for which all the eigenvalues $\lambda^{\Omega_2}_1,\dots, \lambda^{\Omega_2}_{{q_2}+m_2-1}$  for Problem (\ref{mainpsisimp}) in $\Omega_2$ are simple. 

We can iterate this procedure countably many times, yielding a sequence of positive numbers $0<\varepsilon_\ell<    \frac{\varepsilon_\ell}{2^\ell}$, a sequence of perturbations $\psi_\ell\in \mathcal{D}$ with $\|\psi_\ell\|_\mathcal{D}<\varepsilon_\ell$, a sequence of maps $F_\ell:=(I+\psi_\ell)$ and a sequence of domains $\psi_\ell=F_\ell\circ\dots\circ F_1(\Omega)$ for  which all the eigenvalues $\lambda_1^{\Omega_\ell},\dots,\lambda^{\Omega_\ell}_{q_\ell+m_\ell-1}$  for Problem (\ref{mainpsisimp}) in $\Omega_\ell$ are simple. 

To conclude the proof, define $\mathcal{F}_\ell:=F_\ell\circ\dots\circ F_1$. By construction, and due to the choice on the size of $\varepsilon_l$, we have that $\mathcal{F}_\ell\rightarrow \mathcal{F}_\infty$ in $C^2(\mathbb{R}^n,\mathbb{R}^n)$, that $\psi_\infty:=\mathcal{F}_\infty-I$ belongs to $\mathcal{D}$ and that $\|\psi_\infty\|_\mathcal{D}<2\varepsilon_1$. At this point one can easily prove that  the eigenvalues for Problem (\ref{mainpsisimp}) in $\Omega_\infty:=I+\psi_\infty(\Omega)$ are all simple and we conclude the proof. 
\end{proof}

\section{Proof of Theorem \ref{thm:2-simp}}\label{sec:sloshup}
We address equation (\ref{2psisimp}). In this case we use the Hilbert space $H^1(\Omega)$ with the usual scalar product 
\[
a_{\Omega}(u,v):=\int_\Omega \nabla u \nabla v \ dx +\int_{ \Omega} uv\ dx .
\]
which induce the complete norm $\|u\|_{H^1}=\left(a_{\Omega}(u,u)\right)^{1/2}$. By the change of variables $I+\psi$, as before we pull back the bilinear form $a_{\Omega_\psi} $, defined  on $H^1 (\Omega_\psi)$, on $H^1(\Omega)$. In fact, for any $\tilde u,\tilde \varphi\in H^1 (\Omega_\psi)$ we set
\begin{multline}
    a_{\Omega_\psi}(\tilde u,\tilde \varphi)=
   \int_{ \Omega_\psi}\nabla \tilde u\nabla\tilde \varphi\  dx
    +\int_{ \Omega_\psi}\tilde u\tilde \varphi\ dx 
    \\
    =\mathcal{A}_\psi (u,v)
    +\int_{ \Omega} u \varphi J_\psi dx =:\mathcal{E}_\psi (u,\varphi) 
\end{multline}
 where $u=\gamma_{\psi}\tilde u$, $\varphi=\gamma_{\psi} \tilde \varphi$ .   
 As before, for any domain $\Omega$ we define the operator $E_\Omega: H^1(\Omega)\rightarrow H^1(\Omega)$  as the map which associates an element $f\in H^1(\Omega)$ with the unique element in $ H^1(\Omega)$ such that
\begin{equation}
 a_{\Omega}(E_\Omega(f),\varphi)=\int_{\partial \Omega} d(x)f\varphi\  d\sigma 
\end{equation}
holds for all $\varphi \in H^1 (\Omega)$. By the operator $E_\Omega$, problem (\ref{Slosh-up}) is equivalent to find a pair $(e,\lambda)$ for which $e=E_\Omega(\lambda e)$, and an analogous formula holds for  problem (\ref{2psisimp}) on a perturbed domain $\Omega_\psi$.
Finally, we combine the operator $E_{\Omega_\psi}$ with the pushforward and the pullback maps, defining the operator $T_\psi:=\gamma_\psi E_{\Omega_\psi}\gamma_\psi^{-1}$, $T_\psi : H^1(\Omega)\rightarrow H^1(\Omega)$ and we have the fundamental identity 
\begin{equation}
    \mathcal{E}_\psi (T_\psi u,\varphi)=\int_{\partial \Omega} d(x)u\varphi\ B_\psi d\sigma .
\end{equation} 
To find the no splitting condition which holds in this framework, we differentiate the above identity with respect to the $\psi$ variable, at $\psi=0$, obtaining
 
\begin{multline}\label{condthm4}
    a_\Omega  (T_\psi '(0)[\psi] u,\varphi)=
   \int_\Omega\psi_t \frac{\partial}{\partial x_t} \left(\nabla T_0u\nabla \varphi\right) - \psi_i\frac{\partial }{ \partial x_j}\left(\frac{\partial T_0u}{\partial x_i}\frac{\partial \varphi}{\partial x_j}+
\frac{\partial T_0u}{\partial x_j}\frac{\partial \varphi}{\partial x_i}
\right) dx\\ 
-    \int_{\partial\Omega}\psi_t \nu_t \nabla T_0u\nabla \varphi -
\psi_i \nu_j\left(\frac{\partial T_0u}{\partial x_i}\frac{\partial \varphi}{\partial x_j}+\frac{\partial T_0u}{\partial x_j}\frac{\partial \varphi}{\partial x_i}
\right) d\sigma\\
+ \int_\Omega \psi_t \frac{\partial}{\partial x_t}  T_0u \varphi  dx
-\int_{\partial\Omega}\psi_t \nu_t  T_0u \varphi d\sigma
 +\int_{S}  u\varphi \left(\div \psi -\sum_{r=1}^n \partial_\nu \psi_r \nu_r\right)d\sigma  \\
    =-L_1( T_0u, \varphi)-
I_1( T_0u, \varphi)
-I_3(T_0 u, \varphi)+G(u,\varphi),
\end{multline}
thus the no splitting condition is
\begin{equation}\label{nosplithm4}
-\mu L_1( e_r, e_s)-
\mu I_1( e_r, e_s)
-\mu I_3(e_r, e_s)+G(e_r, e_s)=C \delta_{rs}    
\end{equation}
With the same strategy of the previous proof, we choose a perturbation $\psi$ compactly supported in the interior of $\Omega$ and we can absorb the integral on the domain $\Omega$  in the right hand side of (\ref{nosplithm4}), so the no splitting condition becomes
\begin{multline}
-  \mu  \int_{\partial\Omega}\psi_t \nu_t \nabla e_r \nabla e_s + \mu  \int_{\partial\Omega}
\psi_i \left(\frac{\partial e_r}{\partial x_i}\frac{\partial e_s}{\partial \nu}+\frac{\partial e_r}{\partial \nu}\frac{\partial e_s}{\partial x_i}
\right) d\sigma\\
-\mu \int_{\partial\Omega}\psi_t \nu_t  e_re_s d\sigma
 +\int_{S}  e_r e_s \left(\div \psi -\sum_{r=1}^n \partial_\nu \psi_r \nu_r\right)d\sigma  =C \delta_{rs}    
\end{multline}

Since for an eigenvalue $\partial_\nu e_r=d(x)\lambda e_r$, $d=1$ on $S$ and $d=0$ on $W$, we have
\begin{equation*}
    \int_{\partial\Omega}
\psi_i \left(\frac{\partial e_r}{\partial x_i}\frac{\partial e_s}{\partial \nu}+\frac{\partial e_r}{\partial \nu}\frac{\partial e_s}{\partial x_i}
\right) d\sigma=
\lambda \int_{S}
\psi_i  \frac{\partial }{\partial x_i}(e_r e_s) d\sigma=-\lambda \int_{S}e_r e_s
\div\psi d\sigma
\end{equation*} 
and, since $\lambda =1/\mu$, condition (\ref{nosplithm4}) further simplifies to

\begin{equation}\label{nosplitthm4S}
  \mu  \int_{S}\psi_t \nu_t (\nabla e_r \nabla e_s +e_re_s)
 +\int_{S}  e_r e_s\sum_{r=1}^n \partial_\nu \psi_r \nu_rd\sigma  =-C \delta_{rs},    
    \end{equation}
and there exists at least one component $t$, say $t=1$, for which $\nu_t \neq 0$.
Choosing $\psi=(\psi_1,0,\dots,0)$ with $\psi_1$ arbitrarily defined on $S$, we can absorb the first integral in the right hand side of (\ref{nosplitthm4S}), which becomes

\begin{equation*}
    - C(\psi)\delta_{rs}=
     +\mu \int_{ S}  e_r  e_s\sum_{r=1}^n \partial_\nu \psi_r \nu_r d\sigma
\end{equation*}
Since we can choose $\sum_{r=1}^n \partial_\nu \psi_r \nu_r $ as and arbitrary function on $S$, we get that $ e_r  e_s=0$ if $r\neq s$ and $| e_1|^2=\dots=| e_m|^2$ almost everywhere on $S$. This, combined with the boundary condition, gives,  for any $r=1,\dots, m$,
\begin{eqnarray}
    e_r=0;&& \frac{\partial e_r}{\partial \nu}=\lambda e_r=0
\end{eqnarray}
which leads to a contradiction.

\begin{rem*}
    Unlike Theorem \ref{thm:main-simp}, here we are only able to handle perturbations of  $S$. When considering perturbations supported on $W$,  we encounter a technical difficulty. In fact, in this case condition  (\ref{nosplithm4}) becomes
\begin{equation}\label{thm4S}
  \mu  \int_{W}\psi_t \nu_t (\nabla e_r \nabla e_s +e_re_s) =-C \delta_{rs},    
    \end{equation}
Since $\psi$ is arbitrary, and since $\partial_\nu e_r=0$ on $W$, it follows that 
\begin{equation}\label{condHenry}
    \nabla_W e_r\nabla_W e_s =- e_re_s \text{ on }W \text{ for }r\neq s,
\end{equation}
where $\nabla_W$ represents the tangential part of the gradient.  However, at present, we are unable to derive a contradiction from this condition.
\end{rem*}

\section{Problem with coefficients}
The proofs of the theorems in these cases follow a scheme analogous to the case of domain perturbations. Therefore hereafter  we restrict our attention to the variational setting and the resulting no splitting condition, while we refer to the previous sections for the details to achieve a complete proof.

\subsection{No splitting condition for problem (\ref{pot-down})}
Firstly we analyze the case of Problem (\ref{pot-down}). For this problem we call
\begin{equation}
    \label{defApot}
{A} (u,\varphi) :=\int_{\Omega} \nabla u \nabla  \varphi \ dx +\int_{\partial \Omega} a(x) u\varphi d\sigma,   
\end{equation}
which, as $a$ is strictly positive on $\partial \Omega$, is an equivalent norm on $H^1$.
As in the previous sections, we define $E:H^1(\Omega)\rightarrow H^1(\Omega)$ as follows
\begin{equation}\label{eq:Eusimp-down}
 {A} (E(f),\varphi)=\int_{\partial \Omega} d(x)f\varphi\  d\sigma 
\end{equation}
for all $\varphi \in H^1 (\Omega)$, so $e$ is an eigenfunction of Problem (\ref{pot-down})
with eigenvalue $\lambda$, if and only if $e$ is an eigenfunction of $E$ with eigenvalue $\mu:=1/\lambda$.

Finally we define, for $\|b\|_{C^0}$ small, the analogous of ${A} $ and $E$ for the perturbed problem, namely

\begin{equation}
{A}_b(u,\varphi) :=\int_{\Omega} \nabla u \nabla  \varphi \ dx +\int_{\partial \Omega} (a+b)u\varphi,   
\end{equation}
and $E_b$ such that

\begin{equation}\label{eq:idpot}
 {A}_b (E_b(u),\varphi)=\int_{\partial \Omega} d(x)u\varphi\  d\sigma 
\end{equation}

It is easy to see that both $b\mapsto {A}_b$ and $b\mapsto E_b$ are Fr\'echet differentiable in the $b$ variable at $b=0$ and, by (\ref{eq:idpot}) it holds
\begin{equation}
 {A'}_b(0)[b] (E(u),\varphi)+{A} (E'_b(0)[b](u),\varphi)=0 
\end{equation}
Since, by direct computation, $ {A'}_b(0)[b](u,\varphi)=\int_{\partial \Omega} bu\varphi$,  if $\lambda$ is an eigenvalue of Problem (\ref{pot-down}) with multiplicity $m>1$ and $e_i,e_j$ are two elements of the orthonormal base of the eigenspace relative to $\lambda$, we have the no splitting condition
\begin{equation}\label{nosplit-pot-down}
    {A} (E'_b(0)[b](e_i),e_j)=-\mu \int_{\partial \Omega}be_ie_jd\sigma=c\delta_{ij} 
\end{equation}
for all $b$. 

If (\ref{nosplit-pot-down}) is fulfilled for all $b$, then almost everywhere on $\partial \Omega$ $e_i e_j(x)=0$ if $i\neq j$ and $e_1^2(xs)=e_2^2(x)=\dots=e_m^2(x)$. Thus, for all $i$,  $e_i=\partial_\nu e_i=0$ almost everywhere on $\partial \Omega$. This implies $e_i=\nabla e_i=0$  almost everywhere on $\partial \Omega$ and this leads to a contradiction by the unique continuation principle. 
Thus, for all $\varepsilon$, there exists at least a $b(x)$ with $\|b\|_{C^0}<\varepsilon$ such that the no splitting condition  (\ref{nosplit-pot-down}) is not fulfilled. 
\subsection{No splitting condition for problem (\ref{pot-up})}

This case is completely similar, we start defining the equivalent norm on $H^1$ given by 
\begin{equation}
    \label{defApot-d}
{A} (u,\varphi) :=\int_{\Omega} \nabla u \nabla  \varphi \ dx +\int_{\Omega } a(x) u\varphi dx,   
\end{equation}
and $E:H^1(\Omega)\rightarrow H^1(\Omega)$ as
\begin{equation}\label{eq:Eusimp-up}
 {A} (E(f),\varphi)=\int_{\partial \Omega} d(x)f\varphi\  d\sigma 
\end{equation}
for all $\varphi \in H^1 (\Omega)$. Then we set ${A}_b $  and $E_b$ as
\begin{eqnarray}
    {A}_b(u,\varphi) &:=&\int_{\Omega} \nabla u \nabla  \varphi \ dx +\int_{ \Omega} (a+b)u\varphi dx,   \\
    {A}_b (E_b(u),\varphi)&=&\int_{\partial \Omega} d(x)u\varphi\  d\sigma 
\end{eqnarray}
Again, both $b\mapsto {A}_b$ and $b\mapsto E_b$ are Fr\'echet differentiable in the $b$ variable at $b=0$ and, it holds
\begin{equation}
 {A'}_b(0)[b] (E(u),\varphi)+{A} (E'_b(0)[b](u),\varphi)=0 
\end{equation}
Now, if $\lambda$ is an eigenvalue of Problem (\ref{pot-up}) with multiplicity $m>1$ and $e_i,e_j$ are two elements of the orthonormal base of the eigenspace relative to $\lambda$, we obtain the no splitting condition that for Problem (\ref{pot-up}) reads as
\begin{equation}\label{nosplit-pot-up}
    {A} (E'_b(0)[b](e_i),e_j)=-\mu \int_{ \Omega}be_ie_jdx =c\delta_{ij} 
\end{equation}
for all $b$. 
Again, if (\ref{nosplit-pot-up}) is fulfilled for all $b$, then, for all $i$, $e_i=0$ almost everywhere on $\Omega$ and this leads us to a contradiction.

\subsection{Proof of Theorem \ref{thm:pot}}

In the previous paragraph, we showed that, for both Problems (\ref{pot-down}) and (\ref{pot-up}), for any $\varepsilon$, there exists a perturbation $b$, with $\|b\|_{C^0}<\varepsilon$ 
for which lowers the multiplicity of a single eigenvalue. 
By iterating this procedure a finite number of time we show that there exists a $b$ such that a single eigenvalue $\lambda$ splits in $m$ simple eigenvalues. 
We can now conclude, similarly to the other proofs. In fact, suppose that $\lambda^a_{q_1}$ is the first multiple eigenvalue with multiplicity $m_1>1$, there exists an $\varepsilon$ and a $b_1$ with $\|b_1\|_{C^0}<\varepsilon/2$ which splits the eigenvalue $\lambda^a_{q_1}$ in $m_1$ $\lambda^{a+b_1}_{q_1,1}, \lambda^{a+b_1}_{q_1,m_1} $simple eigenvalues.
If there are no more multiple eigenvalues, the proof is complete. 
Otherwise, let  $\lambda^{a+b_1}_{q_2}$ another multiple eigenvalue with multiplicity $m_2>1$. 
Again we can find $b_2$ with $\|b_2\|<\varepsilon/2^2$ and $\lambda^{a+b_1}_{q_2}$ splits in 
$\lambda^{a+b_1+b_2}_{q_2,1}, \lambda^{a+b_1+b_2}_{q_2,m_2} $ simple eigenvalues.

We can iterate this procedure countably many times, choosing a sequence of perturbations $b_l$ 
with $\|b_l\|_{C^0}<\varepsilon/2^l$. By the choice of the norm of these perturbations, 
we have that $b=\sum_{i=1}^{\infty} b_i$ is a perturbation with $\|b\|\le \varepsilon$ 
for which all the eigenvalues are simple.

\subsection{Proof of Theorem \ref{thm:div}}

Being the problem of this theorem completely analogous to the proof of Theorem \ref{thm:pot}, 
we limit ourselves to find the no splitting condition in this setting and to show that there exists at least a $B(x)=(b_{ij}(x))_{ij}$ for which the condition is not verified.
In this case the equivalent norm on $H^1$ is given by 
\begin{equation}
{A} (u,\varphi) :=\int_{\Omega} \sum_{ij} a_{ij}\partial_i u \partial_j  \varphi \ dx +\int_{\partial \Omega } u\varphi d\sigma,   
\end{equation}
and $E$ is given by $ {A} (E(f),\varphi)=\int_{\partial \Omega} d(x)f\varphi\  d\sigma $ for all $\varphi \in H^1 (\Omega)$. 
Again we set ${A}_B $  and $E_B$ as
\begin{eqnarray}
    {A}_B(u,\varphi) &:=&\int_{\Omega} \sum_{ij} (a_{ij}+b_{ij})\partial_i u \partial_j  \varphi \ dx +\int_{\partial \Omega } u\varphi d\sigma  \\
    {A}_B (E_B(u),\varphi)&=&\int_{\partial \Omega} d(x)u\varphi\  d\sigma 
\end{eqnarray}
Differentiating in the $B$ variable at $B=0$ the last equation it holds
\begin{equation}
 {A'}_B(0)[B] (E(u),\varphi)+{A} (E'_B(0)[B](u),\varphi)=0 
\end{equation}
and, if $\lambda$ is an eigenvalue of Problem (\ref{div-down}) with multiplicity $m>1$ and $e_r,e_s$ are two elements of the orthonormal base of the eigenspace relative to $\lambda$, we have the desired no splitting condition that for Problem (\ref{div-down}) 
\begin{equation}\label{nosplit-div-down}
    {A} (E'_B(0)[B](e_r),e_s)=-\mu \int_{ \Omega}
    \sum_{ij}b_{ij} \partial_i e_r \partial_j e_sdx =c\delta_{es} 
\end{equation}
for all $B=(b_{ij})_{ij}$. 
If (\ref{nosplit-div-down}) holds for all $B$, then almost everywhere on $\Omega$, we have that $\partial_i e_r\partial_j e_s=0$ for all $i,j$ and for all $r\neq s$ and that $\partial_i e_1\partial_j e_1=\partial_i e_2\partial_j e_2=\dots =\partial_i e_m\partial_j e_m$   for all $i,j$. 
This implies that $e_1,\dots e_m$ are constant on $\Omega$ and, by the boundary condition, that $e_r=0$ in $\Omega$ for $r=1,\dots, m$, which gives a contradiction.

\appendix

\section{Computation of derivatives}\label{Sec:der}

Hereafter we collect a series of technical Lemma necessary to compute the derivatives of some integral terms, which are needed to prove the results involving domain perturbations.

\begin{lem}\label{lem:Astorto} Set $$
\mathcal{A}_\psi(u,\varphi):
=\int_{\Omega_\psi} \nabla \tilde u\nabla \tilde \varphi d\xi.$$ 
Then, the derivative of $\mathcal{A}$ with respect to $\psi$ at $\psi=0$ satisfies
\begin{eqnarray*}
    \mathcal{A}'_\psi(0)[\psi](u,\varphi)=
&\int_\Omega-\psi_t \frac{\partial}{\partial x_t} \nabla u\nabla \varphi + \psi_i\frac{\partial }{ \partial x_j}\left(\frac{\partial u}{\partial x_i}\frac{\partial \varphi}{\partial x_j}+
\frac{\partial u}{\partial x_j}\frac{\partial \varphi}{\partial x_i}
\right) dx\\
&+\int_{\partial\Omega}\psi_t \nu_t \nabla u\nabla \varphi -
\psi_i \nu_j\left(\frac{\partial u}{\partial x_i}\frac{\partial \varphi}{\partial x_j}+\frac{\partial u}{\partial x_j}\frac{\partial \varphi}{\partial x_i}
\right) d\sigma
\end{eqnarray*} 
\end{lem}
\begin{proof}
Observe that
\begin{displaymath}
    \tilde u(\xi)=(\gamma_\psi^{-1} u)(\xi)=(\gamma_\chi u)(\xi) =u(\xi +\chi \xi), 
\end{displaymath}
so 
$$
\frac{\partial }{\partial \xi_j}\tilde u=\frac{\partial }{\partial \xi_j}u(\xi +\chi \xi)=
\frac{\partial u}{\partial x_i}
\frac{\partial (\xi +\chi \xi)_i}{ \partial \xi_j}
=\frac{\partial u}{\partial x_j}+
\frac{\partial u}{\partial x_i}\frac{\partial \chi_i}{ \partial \xi_j}.
$$
At this point (with the Einstein convention on repeated indexes)
\begin{align*}
\mathcal{A}_\psi(u,\varphi):
=&
\int_{\Omega_\psi} \frac{\partial \tilde u }{\partial \xi_j}\frac{\partial \tilde \varphi }{\partial \xi_j}d\xi
=\int_{\Omega_\psi} 
\left(
\frac{\partial u}{\partial x_j}+
\frac{\partial u}{\partial x_i}\frac{\partial \chi_i}{ \partial \xi_j}
\right)
\left(
\frac{\partial \varphi}{\partial x_j}+
\frac{\partial \varphi}{\partial x_s}\frac{\partial \chi_s}{ \partial \xi_j}
\right)\\
=&\int_\Omega \nabla u \nabla \varphi |J_\psi| dx+
\int_\Omega \frac{\partial u}{\partial x_j}
\frac{\partial \varphi}{\partial x_s}\frac{\partial \chi_s}{ \partial \xi_j}
|J_\psi| dx
+\int_\Omega \frac{\partial u}{\partial x_i}\frac{\partial \chi_i}{ \partial \xi_j}\frac{\partial \varphi}{\partial x_j}|J_\psi| dx\\
&+\int_\Omega \frac{\partial u}{\partial x_i}\frac{\partial \chi_i}{ \partial \xi_j}
\frac{\partial \varphi}{\partial x_s}\frac{\partial \chi_s}{ \partial \xi_j}
|J_\psi| dx\\
=&\int_\Omega \nabla u \nabla \varphi |J_\psi| dx+
\int_\Omega 
\left(\frac{\partial u}{\partial x_i}\frac{\partial \varphi}{\partial x_j}+
\frac{\partial u}{\partial x_j}\frac{\partial \varphi}{\partial x_i}
\right)\frac{\partial \chi_i}{ \partial \xi_j}
|J_\psi| dx\\
&+\int_\Omega \frac{\partial u}{\partial x_i}\frac{\partial \chi_i}{ \partial \xi_j}
\frac{\partial \varphi}{\partial x_s}\frac{\partial \chi_s}{ \partial \xi_j}
|J_\psi| dx
\end{align*}
Since $\|\psi\|_{C^2}<1/2$, we can express $I+\chi=(I+\psi)^{-1}=\sum_{t=0}^\infty (-1)^t(\psi)^t$, then, taking the linear part of the sum, we have the first derivatives of $\chi$, that is $\frac {\partial \chi_i}{\partial \xi_j} =-\frac {\partial \psi_i}{\partial x_j}$.
Also, by elementary computations
\begin{equation}
    \label{derjacobiano}
|\mathrm{det}J_{I+\varepsilon \psi}|=1+\varepsilon \div\psi +o(\varepsilon)
\end{equation}
so $J_\psi'(0)[\psi])=\div\psi$ . With these equalities in mind, we can compute the derivatives of $\mathcal{A}_\psi(u,\varphi)$. In fact we have
\begin{eqnarray*}
    \mathcal{A}'_\psi(0)[\psi](u,\varphi)=&\int_\Omega\div\psi \nabla u\nabla \varphi - \left(\frac{\partial u}{\partial x_i}\frac{\partial \varphi}{\partial x_j}+
\frac{\partial u}{\partial x_j}\frac{\partial \varphi}{\partial x_i}
\right)\frac{\partial \psi_i}{ \partial x_j}dx
\end{eqnarray*} 
and integration by parts yields the desired result.
\end{proof}
\begin{rem*}
    Notice that, when $\varphi$ is compactly supported in $\Omega$, then we simply have 
\begin{equation} \label{suppcomp}
\mathcal{A}'_\psi(0)[\psi](u,\varphi)=
\int_\Omega-\psi_t \frac{\partial}{\partial x_t} \nabla u\nabla \varphi + \psi_i\frac{\partial }{ \partial x_j}\left(\frac{\partial u}{\partial x_i}\frac{\partial \varphi}{\partial x_j}+
\frac{\partial u}{\partial x_j}\frac{\partial \varphi}{\partial x_i}
\right) dx
\end{equation}
    
\end{rem*}
\begin{lem}\label{lem:B}
    Let $B_\psi$ denote the Jacobian of change of variables on $\partial \Omega$ induced by $\psi$. We have 
$$
\left. \frac{\partial}{\partial \psi}\int_{\partial \Omega} u\varphi\ B_\psi d\sigma\right|_{\psi=0}=\int_{\partial \Omega} u\varphi 
\left(
\div \psi -\sum_{r=1}^n 
\partial_\nu \psi_r \nu_r\right)d\sigma
 $$
and
$$
\left. \frac{\partial}{\partial \psi}\int_{\partial \Omega}d(x)  u\varphi\ B_\psi d\sigma\right|_{\psi=0}=\int_{\partial \Omega} d(x)u\varphi
     \left(\div \psi -\sum_{r=1}^n \partial_\nu \psi_r \nu_r\right)d\sigma    
$$where $d:=1_S$
\end{lem}

\begin{proof}
Given $\tilde u, \tilde \varphi\in H^1(\Omega_\psi)$, and set $\nu_{\Omega_\psi}$ the outward normal to $\partial \Omega_\psi$, we extend $\nu_{\Omega_\psi}$ smoothly to the whole $\Omega_\psi$, and we define the vectorial function $\tilde F:\Omega_\psi\rightarrow \mathbb{R}^n$ as $\tilde F:=\tilde u\tilde \varphi \nu_{\Omega_\psi}$.  Now, by the divergence theorem we have the identity
\begin{equation}\label{eq:derbordo0}
    \int_{\Omega_\psi}\frac{\partial}{\partial \xi_i}\tilde F_i d\xi=\int_{\partial \Omega_\psi}\tilde F_i(\nu_{\Omega_\psi})_i d\sigma =
    \int_{\partial \Omega_\psi}\tilde u(\xi) \tilde \varphi (\xi) d\sigma 
\end{equation}
By the usual change of variables we have $$\int_{\partial \Omega_\psi}\tilde u(\xi) \tilde \varphi (\xi) d\sigma= \int_{\partial \Omega} u(x)  \varphi (x) B_\psi d\sigma $$
Thus, differentiating in the $\psi$ variable at $\psi=0$ by  (\ref{eq:derbordo0})  we have 
\begin{equation}\label{eq:derbordo1}
    \left. \frac{\partial}{\partial \psi}\int_{\partial \Omega} u\varphi\ B_\psi d\sigma\right|_{\psi=0}=\left. \frac{\partial}{\partial \psi}
      \int_{\Omega_\psi}\frac{\partial}{\partial \xi_i}\tilde F_i d\xi
    \right|_{\psi=0}.
\end{equation}
Now, calling $F=\gamma_\psi \tilde F$, so $\tilde F (\xi)= \gamma_\psi^{-1} F(\xi)=\gamma_\chi F(\xi)=F(\xi +\chi \xi)$ and proceeding as in Lemma \ref{lem:Astorto}, we obtain 
\begin{equation}\label{eq:derbordo2}
 \int_{\Omega_\psi}\frac{\partial}{\partial \xi_i}\tilde F_i d\xi=
  \int_{\Omega_\psi}\frac{\partial F_i}{\partial x_i}  +\frac{\partial  F_i }{\partial x_t}\frac{\partial \chi_t}{\partial x_i} d\xi=
   \int_{\Omega}\frac{\partial F_i}{\partial x_i}  -\frac{\partial  F_i }{\partial x_t}\frac{\partial \psi_t}{\partial x_i} J_\psi dx
\end{equation}
Combining (\ref{eq:derbordo1}) and (\ref{eq:derbordo2}), and integrating by parts we get
\begin{multline}
     \left. \frac{\partial}{\partial \psi}\int_{\partial \Omega} u\varphi\ B_\psi d\sigma\right|_{\psi=0}\\=
      \left. \frac{\partial}{\partial \psi}
       \int_{\Omega}
       \left(\frac{\partial F_i}{\partial x_i}  -\frac{\partial  F_i }{\partial x_t}\frac{\partial \psi_t}{\partial x_i}\right) J_\psi dx\textbf{}  
        \right|_{\psi=0}
        =\int_\Omega \frac{\partial F_i}{\partial x_i}\div \psi dx-
        \int_\Omega \frac{\partial  F_i }{\partial x_t}\frac{\partial \psi_t}{\partial x_i}dx\\
        =
        -\int_\Omega F_i   \frac{\partial^2 \psi_t}{\partial x_i\partial x_t}dx
        +\int_{\partial\Omega} F_i\nu_i \div \psi d\sigma
        +\int_\Omega  F_i \frac{\partial^2 \psi_t}{ \partial x_t\partial x_i}dx
      -  \int_{\partial \Omega } F_i \nu_t\frac{\partial \psi_t}{\partial x_i}d\sigma\\
      =\int_{\partial\Omega} u\varphi \div \psi d\sigma
      -  \int_{\partial \Omega } u\varphi \nu_i \nu_t\frac{\partial \psi_t}{\partial x_i}d\sigma
\end{multline}
and, since $\nu_i \frac{\partial \psi_t}{\partial x_i}=\partial_\nu \psi_t$, we get the proof.

\end{proof} 
Finally, similarly to the proof of Lemma \ref{lem:Astorto}, we get the last result of this section. 
\begin{lem}\label{lem:uv}
 We have 
\begin{multline*}
\left. \frac{\partial}{\partial \psi}\int_{ \Omega_\psi} \tilde u\tilde \varphi d\xi\right|_{\psi=0}=
\left. \frac{\partial}{\partial \psi}\int_{ \Omega} u\varphi |J_\psi|dx\right|_{\psi=0}
\\=\int_\Omega-\psi_t \frac{\partial}{\partial x_t}  u \varphi  dx
+\int_{\partial\Omega}\psi_t \nu_t  u \varphi d\sigma
 \end{multline*}
\end{lem}

\bibliography{references}
\bibliographystyle{abbrv}

\end{document}